\documentclass[a4paper,leqno]{amsart}

\usepackage{amsmath,amssymb,amsfonts,amsthm,mathrsfs}

\usepackage{hyperref}
\usepackage{pdfsync}

\newcommand{\C}{\mathbb C}

\newcommand{\R}{\mathbb R}
\newcommand{\N}{\mathbb N}

\def\({\left(}
\def\){\right)}
\def\<{\left\langle}
\def\>{\right\rangle}

\def\d{{\partial}}
\def\eps{\varepsilon}
\def\si{\sigma}
\def\le{\leqslant}% lessoreqal
\def\ge{\geqslant}%greaterorequal

\DeclareMathOperator{\IM}{Im}

\def\Eq#1#2{\mathop{\sim}\limits_{#1\rightarrow#2}}
\def\Tend#1#2{\mathop{\longrightarrow}\limits_{#1\rightarrow#2}}

\theoremstyle{plain}
\newtheorem{theorem}{Theorem} [section]
\newtheorem{lemma}[theorem]{Lemma}
\newtheorem{corollary}[theorem]{Corollary}
\newtheorem{proposition}[theorem]{Proposition}

\theoremstyle{remark}
\newtheorem{remark}[theorem]{Remark}

\theoremstyle{definition}
\newtheorem{definition}[theorem]{Definition}

\def\Eq#1#2{\mathop{\sim}\limits_{#1\rightarrow#2}}
\def\Tend#1#2{\mathop{\longrightarrow}\limits_{#1\rightarrow#2}}

\numberwithin{equation}{section}

\begin{document}

\title[Intercritical log-modified Schr\"odinger equation]{On an intercritical log-modified nonlinear Schr\"odinger equation in two spatial dimensions}  

\author[R. Carles]{R\'emi Carles}
\address{Univ Rennes, CNRS\\ IRMAR - UMR 6625\\ F-35000
  Rennes, France}
\email{Remi.Carles@math.cnrs.fr}

\author[C. Sparber]{Christof Sparber}
\address{Department of Mathematics, Statistics, and
  Computer Science\\
  M/C 249\\
University of Illinois at Chicago\\
851 S. Morgan Street
Chicago\\ IL 60607, USA} 
\email{sparber@uic.edu}

\begin{abstract}
We consider a dispersive equation of Schr\"odinger type with a nonlinearity slightly larger than cubic by a logarithmic factor. This equation is 
supposed to be an effective model for stable two dimensional quantum droplets with LHY correction. 
Mathematically, it is seen to be mass supercritical and energy subcritical with a sign-indefinite nonlinearity. For the corresponding initial value problem, we prove global in-time
existence of strong solutions in the energy space. Furthermore, we prove the existence and uniqueness (up to symmetries) of nonlinear ground states 
and the orbital stability of the set of energy minimizers. We also show that for the corresponding model in 1D a stronger stability result is available.
\end{abstract}

\date{\today}

\subjclass[2010]{35Q55, 35A01}
\keywords{Nonlinear Schr\"odinger equation, solitary waves, orbital stability}

\thanks{RC is supported by Rennes M\'etropole through its AIS program. CS acknowledges support by the NSF through grant no. DMS-1348092}
\maketitle

%%%%%%%%%%%%%%%%%%%%%%%%%%%%%%%%%%%%%%%%%%%%%%%%%%%%%%%
%%%%%%%%%%%%%%%%%%%%%%%%%%%%%%%%%%%%%%%%%%%%%%%%%%%%%%%

\section{Introduction}\label{sec:intro}

In this paper we consider the Cauchy problem for the following log-modified nonlinear Schr\"odinger equation (NLS) on $\R^2$:
\begin{equation}\label{eq:nls}
\left\{
\begin{aligned}
 & i\d_t u +\frac{1}{2}\Delta u =\lambda u|u|^2\ln |u|^2,\quad x\in
 \R^2,\ \lambda>0,\\
  & u(0,x) = u_0\in H^1(\R^2).
\end{aligned}  
\right.
\end{equation}
This model is discussed in the physics literature (cf. \cite{MALOMED19, PhysRevLett117, PhysRevLett.123}) as an effective mean-field description of ultra-dilute quantum fluids in two spatial dimensions. 
The logarithmic factor thereby stems from the LHY-correction (after Lee-Huang-Yang), a series expansion in the mean particle density of 
Bose-Einstein condensates with origins in the work of Bogolubov (see, e.g., \cite{LHY, PhysRevLett.115} for more details). 
It is argued that the LHY correction should have a stabilizing effect on an otherwise collapsing condensate, allowing for stable soliton-like modes, 
which are often called {\it quantum droplets}. 
Unfortunately, there are only a few results available to date concerning the rigorous mathematical derivation of the LHY correction, the most recent being \cite{Brietzke2020} 
concerning second order corrections to the (mean-field) bosonic ground state energy in three spatial dimensions. 
The corresponding problem in 2D, however, still remains open.

Nevertheless, the NLS \eqref{eq:nls} has several mathematical
properties which make it an intriguing model to study: It can be seen
as the Hamiltonian evolution equation associated to the  
following {\it energy functional}
\begin{equation}\label{eq:energy}
  E(u) := \frac{1}{2}\|\nabla u\|_{L^2(\R^2)}^2 +\frac{\lambda}{2}
  \int_{\R^2}|u|^4\ln\left(\frac{|u|^2}{\sqrt e}\right)\, dx.
\end{equation}
The latter is thus (at least formally) conserved by solutions to \eqref{eq:nls}, as are the total {\it mass} and {\it momentum}, i.e.,
\begin{equation}\label{eq:mass}
  M(u):=\int_{\R^2}|u|^2\, dx,\quad P(u) := \int_{\R^2}\IM \overline u \nabla u\, dx.
\end{equation}
In view of \eqref{eq:energy}, one sees that the second term in the energy, i.e., the one stemming from the nonlinearity, has no definite sign. Indeed, in terms of the usual classification of NLS (see, e.g. \cite{CazCourant}), the 
nonlinearity in \eqref{eq:nls} is seen to be {\it defocusing} (or
repulsive) whenever the density $|u|^2>\sqrt e$ and {\it focusing} (or
attractive) whenever $|u|^2<\sqrt e$. Furthermore, it is well known that  
in the case of pure power-law nonlinearities such as $\lambda |u|^{p-1}u$, solutions $u$ to NLS obey the additional scaling symmetry
$$
u(t,x)\mapsto u_\mu (t,x)=\mu^{{2}/{(p-1)}} u(\mu^2 t, \mu x),\ \mu >0.
$$
In two spatial dimensions, this implies that the {\it cubic case} $p=3$ is {\it mass-critical}, since in this case the transformation also 
preserves the $L^2(\R^2)$-norm of $u$. It has been proved,
that the corresponding Cauchy problem is globally well-posed in
$L^2(\R^2)$ in the defocusing case, and also in the focusing case for masses below the one of the
ground state (cf. \cite{Dodson15,Dodson16} for more details). 
Furthermore the Cauchy problem becomes ill-posed if
one tries to study it in spaces which are less regular than $L^2$ (\cite{KPV01}).

Coming back to our model, we first note that due to the appearance of the logarithmic factor, \eqref{eq:nls} does not obey any scaling symmetry. However, since 
for all $\eps>0$, we have
\begin{equation*}
  \left\lvert u|u|^2\ln |u|^2\right\rvert\lesssim |u|^{3-\eps} + |u|^{3+\eps} ,
\end{equation*}
the log-modified NLS can formally be seen to be {\it inter-critical}, in two different ways: First, its nonlinearity is slightly larger than cubic, and thus mass supercritical, 
but still remains energy subcritical. Second, it can be understood as the sum of a slightly $L^2$-subcritical (focusing) nonlinearity and a slightly $L^2$-supercritical (defocusing) nonlinearity. It is therefore similar to the case of NLS with competing cubic-quintic power law nonlinearities, i.e. 
\begin{equation}\label{eq:cubicqutinic}
 i\d_t u +\frac{1}{2}\Delta u =- |u|^2u + |u|^4u,\\
 \end{equation}
which has been studied in \cite{KOPV17} in 3D, and, more recently, in \cite{CaSp20, LewinRotaNodari-p} in various space dimensions.
\smallbreak

Our first main result of this work is as follows:

\begin{theorem}[Global well-posedness]\label{thm:gwp}
For any $u_0\in H^1(\R^2)$, there exists a unique global in-time solution 
$u\in C(\R; H^1(\R^2))\cap C^1(\R;H^{-1}(\R^2))$
to \eqref{eq:nls}, depending continuously on the initial data $u_0$. 
Furthermore the solution $u$ obeys the conservation of mass, energy and momentum.
\end{theorem}
 
This result can be interpreted as a rigorous expression of the stabilizing effect of the LHY correction in two spatial dimensions. Recall that the focusing, cubic NLS in 
two spatial dimensions, in general, exhibits finite-time blow-up of solutions. The introduction of the logarithmic factor prevents any such blow-up from happening. 
\begin{remark}
  Since \eqref{eq:nls} is a logarithmic perturbation of the
  $L^2$-critical case, local and global well-posedness might even hold in $L^2(\R^2)$, 
  in view of the similar case of a (smooth) logarithmic perturbation of
  an energy-critical wave-equation considered in \cite{Tao07}. 
\end{remark}
Our second main result concerns the properties of {\it solitary waves}, i.e., solutions of the form 
\[
u(t,x) = e^{i\omega t} \phi(x), \quad \omega \in \R, 
\]
where $\phi$ solves
\begin{equation}
 \label{eq:ground}
 -\frac{1}{2}\Delta \phi +\lambda \phi|\phi|^2\ln |\phi|^2 +\omega
 \phi=0,\quad x\in \R^2.
\end{equation}
Clearly, solutions to this equation can only be unique up to translations and phase conjugation, 
a fact that, together with the Galilei invariance of \eqref{eq:nls}, allows one to subsequently construct 
more general solitary waves, moving with non-zero speed. 

In the following we shall denote the 
{\it action} associated to \eqref{eq:ground} by 
\[
S(\phi) = E(\phi) + \omega M(\phi).
\]
A solution $\phi$ is called a nonlinear \emph{ground state} if it minimizes the action $S(\phi)$ among all possible solutions
 $\phi$ of \eqref{eq:ground}. It follows from \cite[Lemma~2.3]{CiJeSe09} and \cite[Proposition~4]{ByJeMa09} 
 that every minimizer $\varphi$ of the action $S(\phi)$ is of the form
\begin{equation*}
  \varphi(x) = e^{i\theta} \phi_\omega(x-x_0),
\end{equation*}
for some constants $\theta\in \R$, $x_0\in \R^2$, and where $\phi_\omega$ is a {\it positive} least
action solution to \eqref{eq:ground}. The existence and uniqueness of such positive ground states is the content of our second main result.

\begin{theorem}[Existence and uniqueness of positive ground states]\label{thm:ground}
Suppose that the frequency $\omega \in \R$ satisfies
\begin{equation*}
  0<\omega<\frac{\lambda}{2\sqrt e}.
\end{equation*}
Then \eqref{eq:ground} admits a unique solution $\phi_\omega \in C^2(\R^2)$ which is radially symmetric and 
exponentially decaying as $|x|\to \infty$. Moreover, for all $x\in \R^2$, it holds
\[
0<\phi_\omega(x)<\sqrt{z_\omega},   
\]
for some uniquely defined parameter $z_\omega\in (\tfrac{1}{e},1)$, which satisfies $z_\omega \to 1$ as $\omega \to 0_+$.
\end{theorem}
These ground states can be physically interpreted as quantum droplets with zero vorticity. In numerical simulations, 
they are found to have a rather flat top with nearly constant value of the density in its interior, see \cite{MALOMED19}. 

As a final result we shall turn to the question of {\it orbital stability} of solitary waves. To this end we first recall the following notions for constrained energy minimizers.
\begin{definition}\label{def:set-stability}
  For  $\rho>0$, denote
  \begin{equation*}
    \Gamma(\rho) = \left\{ u\in H^1(\R^2),\  M(u)=\rho\right\}.
  \end{equation*}
Assuming that the minimization problem
  \begin{equation}\label{eq:8.3.5}
u\in \Gamma(\rho),\quad E(u)=    \inf \{ E(v)\ ;\  v\in \Gamma(\rho)\} 
  \end{equation}
 has a solution, we shall denote by $\mathcal E(\rho)$ the set of all possible
(constraint) energy minimizers. We call this set {\it orbitally stable}, if 
 for all $\eps>0$, there exists $\delta>0$ such that if
 $u_0\in H^1(\R^2)$ satisfies
 \[\inf_{\phi\in \mathcal E(\rho)}\|u_0-\phi\|_{H^1}\le \delta,\]
  then the
  solution to \eqref{eq:nls} with $u_{\mid t=0}=u_0$ satisfies
  \begin{equation*}
    \sup_{t\in \R}\inf_{\phi\in \mathcal E(\rho)}\left\|u(t,
      \cdot)-\phi\right\|_{H^1}\le \eps. 
   \end{equation*}
\end{definition}

\begin{theorem}[Orbital stability of energy minimizers]\label{thm:stab}
Given any $\rho>0$, the set $\mathcal E(\rho)$
       is non-empty and orbitally stable.
\end{theorem}
The fact that energy minimizers are orbitally stable is in sharp
contrast to the case of the usual focusing cubic NLS in two spatial dimensions, for which
all solitary waves are known to be {\it strongly unstable} due to the possibility of blow-up, see \cite{CazCourant}. 
(In the defocusing case, there is no solitary wave and all solutions scatter.)

Using re-arrangement inequalities, cf. \cite{LiLo}, it is possible to infer that every energy minimizer is radially decreasing and solves \eqref{eq:ground} 
for some Lagrange multiplier $\omega>0$. Hence, the energy minimizer equals a nonlinear ground state $\phi_\omega$, 
possibly after an appropriate space translation (for a given mass and
for a certain fixed $\omega$, minimizing the action or the energy is equivalent). The difficulty, however, is that several $\omega$'s could, at least in principle, yield the same mass $\rho$. 
Thus, uniqueness of solutions to \eqref{eq:ground} at fixed $\omega$ 
does not imply the uniqueness of energy minimizers. The only cases for which this uniqueness is known to be true
seem to be the 
one of a single pure power law nonlinearity $|u|^{p-1}u$, see \cite{CazCourant},
and the one of a purely logarithmic nonlinearity $u\ln|u|^2$, cf. \cite{Ar16}. It is, nevertheless conjectured that uniqueness holds true 
for more general nonlinearities, see e.g. \cite{CaSp20, HajStu, LewinRotaNodari-p} for a more detailed discussion on this. 

The fact that there exists energy minimizer with arbitrarily small mass $\rho>0$ (among the set of ground states), also implies 
that there is no positive lower bound on the mass of ground states. This is in contrast to the case of the 
cubic-quintic NLS \eqref{eq:cubicqutinic} in 2D. 
For the latter it is known that all ground states have mass {\it strictly bigger} than the one of the {\it cubic} nonlinear ground state $Q$, see \cite{CaSp20}.
In Section \ref{sec:prop}, we shall present arguments  showing that
\[
M(\phi_\omega)\equiv\| \phi_\omega\|^2_{L^2}\to 0, \quad \text{as $\omega \to 0_+$.}
\]

The rest of this paper is devoted to the proof of these theorems, which will be done via a series of technical results given in 
Sections \ref{sec:Cauchy}--\ref{sec:stab} below. In there, we will 
also add further remarks and results on topics such as scattering and
the asymptotic behavior of $\phi_\omega$. Finally, in an appendix, we
address the analogue of \eqref{eq:nls} in 1D: Our Theorem~\ref{theo:stab1D} suggests that
ground states for \eqref{eq:nls} are indeed orbitally stable in the sense of,
e.g., \cite{BGR15}.

%%%%%%%%%%%%%%%%%%%%%%%%%%%%%%%%%%%%%%%%%%%%%%%%%%%%%%%%%%
%%%%%%%%%%%%%%%%%%%%%%%%%%%%%%%%%%%%%%%%%%%%%%%%%%%%%%%%%%

\section{Cauchy problem}\label{sec:Cauchy}

\subsection{Global well-posedness} The aim of this subsection is to prove Theorem \ref{thm:gwp}. 
To this end, we start by first proving local well-posedness of \eqref{eq:nls}, when rewritten through Duhamel's formula, i.e.
\begin{equation}\label{eq:duhamel}
u(t) = e^{i\frac{t}{2} \Delta } u_0 - i \lambda \int_0^t e^{i\frac{t-s}{2} \Delta} f(u)(s) \, ds,
\end{equation}
where here and in the following, we denote
\[
f(z)= z|z|^2\ln |z|^2, \quad z\in \C.
\]
A classical fixed point argument, based on the use of Strichartz estimates, then yields the following result.
\begin{proposition}[Local well-posedness]
For any $u_0\in H^1(\R^2)$ and any $\lambda\in \R$, there exist times $T>0$ and a unique solution 
\[
u\in C([0, T];H^1(\R^2))\cap C^1((0, T); H^{-1}(\R^2)), 
\]
to \eqref{eq:duhamel}, depending continuously on $u_0$. Moreover $u$ conserves its mass, energy, and momentum, and 
we also have the blow-up alternative, i.e. if $T<\infty$, then
\[
\lim_{t\to T_-} \| u(t, \cdot)\|_{H^1}=\infty.
\]
\end{proposition}

In view of the fact that \eqref{eq:nls} is time-reversible, we also obtain the analogous statement backward in time.

\begin{proof}
We see that our nonlinearity $f\in C^1(\R^2;\R^2)$ satisfies $f(0)=0$,
\[
| f(u)| \lesssim |u|^{3-\eps} + |u|^{3+\eps} , \quad \forall \eps>0,
\]
as well as
\[
|\nabla f(u)|\le (3 |\ln |u|^2 | +2) |u|^2 |\nabla u| \lesssim (|u|^{2+\eps} + |u|^{2-\eps}) |\nabla u|.
\]
We therefore can simply quote classical results by Kato, in particular
\cite[Theorem~I]{Kato1987} (see also \cite{CazCourant}), to obtain existence and uniqueness of a
strong  
solution $u(t, \cdot)
\in H^1(\R^2)$ to \eqref{eq:duhamel}, up to some (possibly finite) time $T=T(\|u_0\|_{H^1})>0$. 

The proof of the conservation laws for mass, energy and momentum follows along the same lines as in \cite[Theorem~III]{Kato1987} 
(see also \cite{Oz06} for an alternative approach which does not require any additional smoothness of the solution $u$).
\end{proof}
\begin{remark}
It is not clear whether the solution is  arbitrarily smooth or not, in
general, since one can see that the  
third derivative of $f(z)$ becomes singular. See also \cite{CaGa18} in the
case of the (even more singular) nonlinearity $z\ln|z|^2$. 
\end{remark}

\begin{corollary}[Global well-posedness] 
Let $\lambda \ge 0$. Then, the solution is global, i.e. $T=\infty$. 
\end{corollary}

\begin{proof} Using the conservation laws of mass and energy, together with the fact that $\lambda \ge 0$, 
the positive part of the energy satisfies
\begin{align*}
  E_+(u) :&= \frac{1}{2}\|\nabla u(t, \cdot)\|_{L^2}^2 +\frac{\lambda}{2}
           \int_{|u|^2>\sqrt e}|u(t,x)|^4\ln\(\frac{|u(t,x)|^2}{\sqrt e}\)dx\\
  &=
  E(u_0)+ \frac{\lambda}{2}
    \int_{|u|^2<\sqrt e}|u(t,x)|^4\ln\(\frac{\sqrt e}{|u(t,x)|^2}\) dx\\
  &\le E(u_0) +\frac{\lambda}{2}  \int_{\R^2} |u(t,x)|^4\(\frac{\sqrt
    e}{|u(t,x)|^2}\)dx = E(u_0) +\frac{\lambda}{2}  \sqrt e M(u_0).
\end{align*}
This consequently yields a uniform in-time 
bound on $\|u(t, \cdot)\|_{H^1}$ and thus, the blow-up alternative implies that $T=\infty$.  
\end{proof}

\subsection{Some scattering results} Let us introduce the conformal space 
\[
\Sigma:=\left\{f\in H^1(\R^2),\ x\mapsto |x|f(x)\in
  L^2(\R^2)\right\},\quad  \|f\|_\Sigma =
\|f\|_{H^1(\R^2)}+\left\||x|f\right\|_{L^2(\R^2)}. 
\]

\begin{lemma} Let $u_0\in \Sigma$ and $\lambda \ge 0$, 
then the global in-time solution $u$ obtained above satisfies $u\in C(\R; \Sigma)$.
\end{lemma}

\begin{proof} We introduce the Galilean operator $J(t)=x+it\nabla$, 
which commutes with the free Schr\"odinger equation, i.e.
\begin{equation*}
  \big[J,i\d_t+\tfrac{1}{2}\Delta\big]=0.
\end{equation*}
A direct computation then yields the pseudo-conformal conservation law
\begin{equation*}
  \frac{d}{dt}\(\frac{1}{2}\|(x+it\nabla )u\|_{L^2}^2 +\frac{\lambda
    t^2}{2}\int_{\R^2} |u(t,x)|^4 \ln \(\frac{|u(t,x)|^2}{\sqrt e}\)dx \)
  =-\lambda t \int_{\R^2}|u(t,x)|^4dx.
\end{equation*}
In particular if $\lambda \ge 0$, the same type of argument as in the proof above yields 
that $\|(x+it\nabla )u\|_{L^2}$ is uniformly bounded for all $t\ge 0$. A triangle inequality then implies that $u(t, \cdot)\in \Sigma$.
\end{proof}

\begin{proposition}
\emph{Existence of wave operators:}  If $u_-\in \Sigma$, then there exist $u_0\in \Sigma$ and  $u\in
  C(\R;\Sigma)$ solving \eqref{eq:nls} such that
\begin{equation*}
 \left\| e^{-i\frac{t}{2}\Delta} u(t, \cdot)-
      u_-\right\|_\Sigma \Tend t {-\infty} 0. 
\end{equation*}
 \emph{Small data scattering:} If $u_0\in \Sigma$ and $\|u_0\|_\Sigma$ is sufficiently small, then there exists $u_+\in \Sigma$, such that
\begin{equation*}
 \left\| e^{-i\frac{t}{2}\Delta} u(t, \cdot)-
      u_+\right\|_\Sigma \Tend t \infty 0. 
\end{equation*}
\end{proposition}
\begin{proof}[Sketch of the proof]
Recall that 
\begin{equation*}\label{eq:factor}
 J(t)u = it\, e^{i|x|^2/(2t)}\nabla\(u e^{-i|x|^2/(2t)}\),
\end{equation*}
which implies that $J(t)u$ can be estimated like $\nabla u$ in $L^p$. Using this, one obtains the Gagliardo--Nirenberg type inequality 
adapted to $J(t)$, i.e.
\begin{equation*}
  \|u\|_{L^p(\R^2)}\lesssim
  \frac{1}{t^{1-2/p}}\|u\|_{L^2(\R^2)}^{2/p}
  \|(x+it\nabla )u\|_{L^2(\R^2)}^{1-2/p},\quad 2\le p<\infty.
\end{equation*}
Essentially the same fixed point argument as the one used in solving the Cauchy
problem locally in-time then yields the  existence of wave operators
(see e.g. \cite{CazCourant}). Small data scattering then follows directly from \cite[Theorem~2.1]{NakanishiOzawa}. 
\end{proof}

\begin{remark}
  The existence of wave operators under the mere assumption $u_-\in H^1(\R^2)$ is very
  delicate, since the present nonlinearity can be understood as the
  sum of a slightly $L^2$-subcritical (focusing) nonlinearity and a
  slightly $L^2$-supercritical (defocusing) nonlinearity. The
  existence of wave operators in $H^1$ is known for
  $L^2$-supercritical defocusing nonlinearities, but not for
  $L^2$-subcritical ones. Also, the smallness in $\Sigma$ is necessary
  to have scattering, in the sense that smallness in $H^1(\R^2)$ is not
  enough, see also Remark~\ref{rem:smallness}. 
\end{remark}

%%%%%%%%%%%%%%%%%%%%%%%%%%%%%%%%%%%%%%%%%%%%%%%%%%%%%%%
%%%%%%%%%%%%%%%%%%%%%%%%%%%%%%%%%%%%%%%%%%%%%%%%%%%%%%%

\section{Nonlinear ground states}\label{sec:ground}

\subsection{Necessary and sufficient conditions for the existence of ground states} 
We seek solutions to \eqref{eq:nls} in the form $u(t,x) =e^{i\omega t}\phi(x)$, with $\omega \in \R$ and $\phi$ sufficiently smooth and localized. 
Then $\phi$ solves
\begin{equation}
  \label{eq:soliton}
  -\Delta \phi =g(\phi),\quad \text{on $\R^2$},
\end{equation}
where here, and in the following, we shall denote (in agreement with the notations from \cite{BGK83,BL83a}):
\begin{equation}\label{eq:g}
  g(\phi) = -2\omega \phi -2\lambda|\phi|^2\phi\ln|\phi|^2,\quad G(z):=
  \int_0^z g(s)\, ds.
\end{equation}
We also define the quantities
\begin{equation*}
  T(\phi) := \int_{\R^2}|\phi(x)|^2dx,\quad V(\phi) := \int_{\R^2}G\(\phi(x)\)dx,
\end{equation*}
which allow us to rewrite the Lagrangian action as
\begin{equation}\label{eq:lagrange}
  S(\phi) = \frac{1}{2}T(\phi) -V(\phi).
\end{equation}

In a first step, we shall derive certain necessary conditions for solution $\phi$ to \eqref{eq:soliton}.

\begin{lemma}[Pohozaev identities] Any solution $\phi \in H^1(\R^2)$ to \eqref{eq:soliton} satisfies
\begin{equation}
  \label{eq:phi1}
  \frac{1}{2}\int_{\R^2}|\nabla \phi|^2\, dx
  +\lambda\int_{\R^2}|\phi|^4\ln|\phi|^2 \, dx+\omega \int_{\R^2}|\phi|^2 \, dx= 0,
\end{equation}
as well as
 \begin{equation} \label{eq:phi2}
  \frac{1}{2}\int_{\R^2}|\nabla
   \phi|^2\, dx+\frac{\lambda}{2}\int_{\R^2}|\phi|^4 \, dx= \omega \int_{\R^2}|\phi|^2\, dx.
\end{equation}
Moreover, in order to have a nontrivial solution $\phi\not \equiv 0$, a necessary condition on the frequency $\omega \in \R$ is
\begin{equation*}
  0<\omega<\frac{\lambda}{2\sqrt e}.
\end{equation*}
\end{lemma}

\begin{proof}
First, assume that $\phi$ is sufficiently smooth and rapidly decaying as $|x|\to \infty$. Then we directly obtain \eqref{eq:phi1} by multiplying \eqref{eq:soliton} with $\bar \phi$ and integrating w.r.t. $x\in \R^2$. 
To obtain \eqref{eq:phi2}, we instead multiply by \eqref{eq:soliton} with $x\cdot \nabla \bar \phi$. Integration in $x$ then yields
\begin{equation}
  \label{eq:phi2a}
 \frac{\lambda}{2}\int_{\R^2}|\phi|^4\ln|\phi|^2\, dx
 -\frac{\lambda}{4}\int_{\R^2}|\phi|^4\, dx+\omega \int_{\R^2}|\phi|^2\, dx= 0,
\end{equation}
or, in other words, $V(\phi)=0$.
By taking \eqref{eq:phi1}$-2\times$\eqref{eq:phi2a} we infer \eqref{eq:phi2} for sufficiently ``nice" $\phi$, 
and a limiting argument allows us to extend this result to general $\phi \in H^1(\R^2)$. 
In particular, \eqref{eq:phi2} also implies that $\omega >0$ is necessary for nontrivial $\phi$.

Next, we consider, for $\eps>0$:
\begin{equation*}
  c_\eps = \sup_{0<z<1}z^\eps \ln \frac{1}{z}.
\end{equation*}
Introducing $f_\eps(z) =z^\eps \ln \frac{1}{z}$ and computing its
derivative, we find that
\begin{equation*}
  c_\eps = f_\eps\( e^{-1/\eps}\) = \frac{1}{\eps e}.
\end{equation*}
Taking $\eps=1$, we infer
\begin{equation*}
 0\le \int_{|\phi|^2<\sqrt{e}}|\phi|^4\ln\frac{\sqrt e}{|\phi|^2}\, dx\le
 \frac{1}{\sqrt e} \int_{\R^2} |\phi|^2 \, dx. 
\end{equation*}
Using this within the Pohozaev identity \eqref{eq:phi2a}, which we can be rewritten as
\[
\frac{\lambda}{2}\int_{\R^2}|\phi|^4\ln\frac{|\phi|^2}{\sqrt e}\, dx+\omega \int_{\R^2}|\phi|^2\, dx= 0,
\]
then yields
\begin{equation*}
  \frac{\lambda}{2}\int_{|\phi|^2\ge \sqrt e}
  |\phi|^4\ln\frac{|\phi|^2}{\sqrt e} \, dx+\omega
  \int_{\R^2}|\phi|^2  \, dx= \frac{\lambda}{2}\int_{|\phi|^2< \sqrt e }
  |\phi|^4\ln\frac{\sqrt e}{|\phi|^2}\, dx\le \frac{\lambda}{2\sqrt e}\| \phi\|_{L^2}^2.
\end{equation*}
Since the l.h.s. is the sum of two positive terms (unless $\phi \equiv 0$), this yields the condition that $0<\omega<\tfrac{\lambda}{2\sqrt e}$. 
\end{proof}
Next, we shall show that the necessary condition on $\omega$ obtained above is also sufficient for the existence of positive ground states.

\begin{proposition}[Existence of ground states]\label{prop:ex}
Let $0<\omega<\tfrac{\lambda}{2\sqrt e}$.
Then \eqref{eq:soliton} has a solution $\phi_\omega$, such that:
\begin{enumerate}
\item $\phi_\omega>0$ on $\R^2$.
  \item $\phi_\omega$ is radially symmetric, i.e., $\phi_\omega=\phi_\omega(r)$ with $r=|x|$, and non-increasing.
  \item $\phi_\omega\in C^2(\R^2)$.
    \item The derivatives of $\phi_\omega$ up to order two decay
      exponentially, i.e.,
      \begin{equation*}
        \exists \delta>0,\quad |\d^\alpha \phi_\omega(x)|\lesssim
        e^{-\delta|x|},\quad |\alpha|\le 2.
      \end{equation*}
    \item For every solution $\varphi$ to \eqref{eq:soliton}, we have
      \begin{equation*}
        0< S(\phi_\omega)\le S(\phi),
      \end{equation*}
 where $S$ is the Lagrangian defined in \eqref{eq:lagrange}.     
\end{enumerate}
\end{proposition}
\begin{proof} 
With the exception of the exponential decay asserted in $(4)$, this statement is a direct quotation of \cite[Th\'eor\`eme~1]{BGK83}. We therefore only need to check that 
the function $g$, defined in \eqref{eq:g}, satisfies the conditions $(g.0)-(g.3)$ imposed in \cite{BGK83}.
To this end, we first note that the function $g\in C(\R;\R)$ is obviously odd, and that
 \begin{equation*}
   \lim_{s\to 0} \frac{g(s)}{s}=-2\omega<0,\quad \text{since }\omega>0.
 \end{equation*}
Thus $(g.0)$ and $(g.2)$ are indeed satisfied. In addition, we see that that $g$ is sub-exponential at infinity, hence satisfying condition $(g.3)$. It remains to check $(g.2)$: an integration
 by parts yields, for $z>0$,
 \begin{equation*}
   G(z) = -\omega z^2 -4\lambda\int_0^z s^3\ln s \, ds =  -\omega z^2
   -\lambda z^4\ln z
   +\lambda \frac{z^4}{4}=-\omega z^2 -\frac{\lambda}{2} z^4\ln
   \frac{z^2}{\sqrt e}. 
 \end{equation*}
The map $z\mapsto \tfrac{z^2}{4}- z^2\ln z$ reaches its maximum at $z^*=
e^{-1/4}$, and
\begin{equation*}
  G\(e^{-1/4}\) =\frac{1}{\sqrt e}\( -\omega +\frac{\lambda}{2\sqrt e}\)>0,
\end{equation*}
by our assumption on $\omega$. Therefore, also $(g.2)$ is satisfied and we obtain our result. 
Finally, the exponential decay of the solution (together with its derivatives) follows from standard arguments for ordinary differential equations, see, e.g., \cite[Section 4.2]{BL83a}. 
\end{proof}

\subsection{Uniqueness and further properties}\label{sec:prop}

Having obtained existence of nonlinear ground states, we shall now derive further properties for them.

\begin{lemma}[$L^\infty$-bound]
Let $\phi_\omega$ be a nonlinear ground state. 
Then there exists a unique $z_\omega \in (\tfrac{1}{e}, 1)$, satisfying $z_\omega \to 1$ as $\omega\to 0_+$, such that 
\[
0<\phi_\omega(x)<\sqrt{z_\omega},  \ \text{for all $x\in \R^2$,} 
\]
\end{lemma}

\begin{proof} In view of Proposition \ref{prop:ex}, we know that $\phi_\omega=\phi_\omega(r)>0$ reaches its
  maximum at zero, $\Delta \phi_\omega(0)\le 0$, thus
  \begin{equation*}
    \lambda \phi_\omega^3\ln \phi_\omega^2 +\omega {\phi_\omega}_{\mid r=0}\le 0.
  \end{equation*}
  Therefore, since $\phi_\omega(0)>0$,
  \begin{equation*}
 z\ln z \le -\frac{\omega}{\lambda},\quad \text{where} \quad z=\phi_\omega(0)^2.
  \end{equation*}
  The map $z\mapsto z\ln z$ is negative exactly on $(0,1)$, and
  reaches its minimum value $-\tfrac{1}{e}$ at $z_\ast=\tfrac{1}{e}$. Since $\omega \in (0,\tfrac{\lambda}{2\sqrt e})$ 
  by assumption, there exists a unique $z_\omega\in (\tfrac{1}{e},1)$  
  such that
  \begin{equation*}
z_\omega\ln z_\omega = -\frac{\omega}{\lambda},
  \end{equation*}
and $z_\omega\to 1$ as $\omega\to 0$. 
\end{proof}

\begin{remark} The proof can be generalized to any sufficiently smooth solution $\phi$ to \eqref{eq:soliton}, not necessarily radial and decreasing. 
Indeed, the same argument as above shows that at any point $x_0\in \R^2$
  where $|\phi|$ reaches its maximum: $|\phi(x_0)|\le \sqrt{z_\omega}$. Hence,
  the above estimate generalizes to
  \begin{equation*}
   |\phi(x)|\le \sqrt{z_\omega},\quad \forall x\in \R^2,
 \end{equation*}
 as soon as $\phi\in C^2(\R^2)$ solves \eqref{eq:soliton}. 
    In particular, $|\phi(x)|<1$ for all $x\in \R^2$, hence $\ln
    |\phi|^2<0$, i.e., the nonlinearity can be considered fully focusing. 
\end{remark}

We now turn to the question of uniqueness of nonlinear ground states.

\begin{lemma}[Uniqueness]
There exists at most one positive solution $\phi_\omega$ to \eqref{eq:soliton}.
\end{lemma}

\begin{proof}
This result follows from \cite[Theorem 1.1]{JANG2010} provided we can check the condition $(f1)-(f3)$ imposed on $g$. In view of \eqref{eq:g}, we 
see that $g(0)=0$ and continuous on $[0, \infty)$. Recall that its anti-derivative is
\[
G(z) = \lambda z^2\left(\frac{z^2}{4}
   - z^2\ln z - \tfrac{\omega}{\lambda}\right) \equiv \lambda z^2 \tilde g(z).
\]
A straightforward calculation shows that 
$\tilde g$ is strictly increasing on $[0, e^{-1/4})$ and strictly decreasing on $(e^{-1/4}, \infty)$. 
In addition, we know that 
\[
\tilde g(0)=-\tfrac{\omega}{\lambda}<0,\quad  
\tilde g\(e^{-1/4}\)>0, \quad \text{and $\tilde g(z)\to -\infty$, as $z\to +\infty$.}
\] 
Thus, we can choose $u_1$ as the unique zero of $\tilde g$ on the interval $[0, e^{-1/4})$. 
Furthermore we claim that we can choose $\bar u = \sqrt{z_\omega}$. To this end, one first checks that there exists a unique $\alpha \in [0, e^{-1/2})$, 
such that $g(0)=g(\alpha)=g(\sqrt{z_\omega})=0$ and 
\[
g(z)< 0 \ \text{on $[0, \alpha) \cup (\sqrt{z_\omega}, \infty)$, while} \ g(z)>0 \ \text{on $(\alpha , \sqrt{z_\omega})$.}
\]
By the choice of $u_1$, we have that 
\[
G(u_1)=\int_0^{u_1} g(z) \, dz = 0,
\]
and hence $\alpha < u_1$. In particular, since $g(z)>0$ on $(\alpha , \sqrt{z_\omega})$, this implies that $G(z)>0$ on $(u_1, \sqrt{z_\omega})$.

Finally, to satisfy condition $(f3)$, one needs to check if $s(z)= \frac{z g'(z)}{g(z)}$ is decreasing on $[u_1, \sqrt{z_\omega})$. This follows from a 
lengthy calculation which shows that 
\[
g(z)^2s'(z) = 4\lambda z^3\left(2{\omega} (1+\ln z)-\lambda z^2\right)<0, 
\]
on $(u_1 , \sqrt{z_\omega})$. We therefore have all the necessary ingredients to conclude uniqueness of the ground state.
\end{proof}

The proof Theorem \ref{thm:ground} is now complete.
\smallbreak

{\bf Asymptotics for $\omega \to 0$.}
To show that $M(\phi_\omega)\to 0$, as $\omega \to 0$, 
one can follow the ideas in \cite{KOPV17} for the cubic-quintic case
(see also \cite{MorozMuratov2014}). In there, the asymptotic regime 
$\omega\to 0$ is analyzed through the rescaling
\begin{equation*}
  \psi_\omega(x) =\frac{1}{\sqrt\omega}\phi_\omega\(\frac{x}{\sqrt\omega}\),
\end{equation*}
which is $L^2(\R^2)$-unitary. One finds that $\psi_\omega$ solves
\begin{equation*}
  -\Delta \psi_\omega +\omega \lambda \psi_\omega^5 -\lambda \psi_\omega^3+\psi_\omega=0,
\end{equation*}
and thus, the limit $\omega\to 0$ is no longer singular. Moreover, in
the 2D case,
\[
M(\psi_\omega) = M(\phi_\omega) \Tend  \omega 0 M(Q), 
\]
where $Q$ is the cubic ground state solution to
\begin{equation*}
  -\frac{1}{2}\Delta Q -\lambda Q^3 +Q=0,
\end{equation*}
In our case, the logarithm is not
compatible with such a rescaling. 
Instead, we define 
\begin{equation*}
  \psi_\omega(x)
  =\sqrt{\frac{\ln\frac{1}{\omega}}{\omega}}\phi_\omega\(\frac{x}{\sqrt\omega}\), 
\end{equation*}
and a computation shows that $\psi_\omega$ solves
\begin{equation*}
  -\frac{1}{2}\Delta \psi_\omega -\lambda \psi_\omega^3+
  \psi_\omega = \lambda \frac{\ln \ln
    \frac{1}{\omega}}{\ln\frac{1}{\omega}} \psi_\omega^3
  -\frac{\lambda}{\ln\frac{1}{\omega}} \psi_\omega^3\ln \psi_\omega^2 .
\end{equation*}
Recalling that, as $\omega \to 0$
\[
1\ll \ln \ln\frac{1}{\omega}\ll \ln\frac{1}{\omega} ,
\]
and using the analyticity of $\psi_\omega$ in $\omega$, we have $\psi_\omega \Eq \omega 0 Q$, and thus, in terms of $\phi_\omega$, 
\begin{equation*}
    \phi_\omega(x)\Eq \omega 0 \sqrt{\frac{\omega}{\ln
        \frac{1}{\omega}}}Q(x\sqrt\omega). 
  \end{equation*}
In turn, this implies that
\[
M(\phi_\omega) = \frac{1}{\sqrt{\ln
        \frac{1}{\omega}}} M(Q) \Tend  \omega 0 0.
\]
These formal arguments can be made rigorous by following the steps in \cite{KOPV17}, which are based on the use of the linearized operator
\[L: f \mapsto -\frac{1}{2}f-3\lambda Q^2f+f,\]
which is known to be an isomorphism $L: H^1_{\rm rad}\to H^{-1}_{\rm rad}$, cf. \cite{Weinstein85}. The implicit function theorem then allows one to write 
$\psi_\omega$ in terms of $Q$ plus lower order corrections involving
$L^{-1}$. In the present case, the situation is similar, for the
spectral analysis presented in \cite{KOPV17} is readily adapted to the
present case. Details are left to the interested reader.

\begin{remark}\label{rem:smallness}
  This computation also shows that the $L^\infty$-bound derived before is
  far from being sharp for small $\omega$.
  The fact that the $L^2$-norm of $\phi_\omega$
  can be arbitrarily small, is  in sharp contrast with the cubic-quintic
  case. Also, \eqref{eq:phi2} shows that the $H^1$-norm of
  $\phi_\omega$ 
  can be arbitrarily small: smallness in $H^1$ does not guarantee
  scattering. The smallness of the momentum in
  \cite[Theorem~2.1]{NakanishiOzawa} (and thus $\|u_0\|_\Sigma$ sufficiently
  small) must be considered as necessary (since $\phi_\omega$ decays
  exponentially). 
\end{remark}

%%%%%%%%%%%%%%%%%%%%%%%%%%%%%%%%%%%%%%%%%%%%%%%%%%%%%%%
%%%%%%%%%%%%%%%%%%%%%%%%%%%%%%%%%%%%%%%%%%%%%%%%%%%%%%%

\section{Orbital Stability}\label{sec:stab}

We start by recalling that for $\rho>0$:
  \begin{equation*}
    \Gamma(\rho) = \left\{ u\in H^1(\R^d),\  M(u)=\rho\right\},
  \end{equation*}
and first prove that the constrained energy is bounded below.
\begin{lemma}[Bound on the energy]\label{lem:min} 
For any $\rho>0$,
\begin{equation*}
 \inf\left\{E(u)\, ;\, u\in \Gamma(\rho)\right\} =-\nu,
\end{equation*}
for some finite $\nu>0$. \end{lemma}

\begin{proof} We can estimate
\begin{align*}
E(u) \ge & \, \frac{1}{2}\|\nabla u\|_{L^2}^2 -\frac{\lambda}{2}
           \int_{|u|^2<\sqrt e}|u|^4\ln\(\frac{\sqrt e}{|u|^2}\) dx\\
         \ge & \, \frac{1}{2}\|\nabla u\|_{L^2}^2 -\frac{\lambda \sqrt{e}}{2}
           \int_{\R^2} {|u|^2}\, dx  \ge  \, \frac{1}{2}\| u\|_{H^1}^2-K,
\end{align*}
where $K= \tfrac{\rho}{2}(1+ {\lambda \sqrt{e}})>0$. Thus, all (constrained) energy-minimizing sequences are bounded in $H^1(\R^2)$ 
and $-\nu \ge -K>-\infty$. Moreover, for $\mu>0$, let 
\[
u_\mu (x):= \mu u(\mu x) \quad \text{such that $\| u_\mu\|_{L^2(\R^2)}= \| u\|_{L^2(\R^2)}$.}
\]
Then, one finds that
\[
E(u_\mu) = \mu^2 E(u) - \mu^2 \lambda \ln \( \frac{1}{\mu^2}\) \int_{\R^2} |u|^4 \, dx.
\]
Hence, $E(u_\mu)<0$ for $\mu>0$ sufficiently small, and thus $\nu >0$.
\end{proof}

We shall now show that energy minimizers indeed exist, and that they are orbitally stable (as a set), 
by invoking the concentration-compactness method of 
\cite{PL284a} (see also \cite[Proposition~1.7.6]{CazCourant}.

\begin{proof}[Proof of Theorem \ref{thm:stab}]
We proceed in several steps:
\smallbreak

\noindent {\bf Step 1.} Let $(u_n)_{n\ge 0}\subset H^1(\R^2)$ be a minimizing sequence to \eqref{eq:8.3.5}. In view of \cite{PL284a}, we have the standard
trichotomy of concentration compactness. To rule out vanishing of the sequence, we first note that for
$n$ sufficiently large, Lemma~\ref{lem:min} implies that $E(u_n)\le
-\tfrac{\nu}{2}$, and hence, from the proof of Lemma~\ref{lem:min},
\begin{equation*}
           \int_{|u_n|^2<\sqrt e}|u_n|^4\ln\(\frac{\sqrt e}{|u_n|^2}\) dx\ge   \frac{\nu}{\lambda}>0.
\end{equation*}
In addition,
\[
 \int_{|u_n|^2<\sqrt e}|u_n|^3\, dx \gtrsim \int_{|u_n|^2<\sqrt e}|u_n|^4\ln\(\frac{\sqrt e}{|u_n|^2}\) dx,
\]
and, thus, any minimizing sequence is bounded away
from zero in $L^3(\R^2)$. 
\smallbreak

\noindent {\bf Step 2.} Next, we need to rule out dichotomy, in order to conclude 
compactness. Arguing by contradiction, suppose that, after the extraction of some suitable subsequences,
there exist $(v_k)_{k\ge 0}$, $(w_k)_{k\ge 0}$ in $H^1(\R^2)$, such
that
\begin{equation*}
  \operatorname{supp} v_k\cap \operatorname{supp} w_k=\emptyset,
\end{equation*} 
as well as the following properties:
\begin{align*} 
  \|v_k\|_{L^2}^2\Tend k \infty \theta\rho,\quad
    \|w_k\|_{L^2}^2\Tend k \infty (1-\theta)\rho,\quad\text{for some
    }\theta\in (0,1),
  \end{align*}
\begin{equation}\label{eq:kinetic}
\liminf_{k\to \infty} \(\int |\nabla u_{n_k}|^2 - \int|\nabla
  v_k|^2 -\int |\nabla  w_k|^2 \)\ge 0,
\end{equation}
and the remainder $r_k:= u_{n_k} -v_k -w_k$ satisfies
\[
 \| r_k\|_{L^p}\Tend k \infty 0,
\]
for all $2\le p<\infty$. Note that this also implies
\[
 \left|\int |u_{n_k}|^p -\int|v_k|^p- \int|w_k|^p\right|\Tend k \infty0,
\]
since $v_k$ and $w_k$ have disjoint support.

Denote $h(y)= y^4 \ln y$ for $y>0$. A Taylor expansion on $h(y+z)-h(y)-h(z)$, combined with an induction step shows that 
for $\eps>0$ and $N\ge 1$, that there exists a $C_{\eps, N}>0$, such that
\[
\left| h\( \sum^N_{n =1} y_n \) - \sum^N_{n =1}  h(y_n)\right| \le C_{\eps, N} \sum_{\ell \not = k}^N |y_\ell| \( |y_k|^{3-\eps} + |y_k|^{3+\eps}\).
\]
Applying this with $\eps =1$ and $N=3$ to $v_k, w_k, r_k$, and integrating over $\R^2$, yields
\begin{align*}
& \left | \int h(u_{n_k}) - \int h(v_k) -\int h(w_k) \right| \lesssim \int h(r_k)  \, +\\
& +  \int |r_k| \( |v_k|^{2} + |v_k|^{4} +|w_k|^{2} + |w_k|^{4}\) 
+ \int (|v_k| + |w_k| )\( |r_k|^{2} + |r_k|^{4}\),
\end{align*}
where in the second line we have used the fact that $v_k$ and $w_k$ have disjoint supports. Applying H\"older's inequality 
and recalling that $\|r_k\|_{L^p}\to 0$, as $k \to \infty$, shows that all the integrals on the right hand side tend to zero in the limit $k\to \infty$, hence
\[
 \left | \int h(u_{n_k}) - \int h(v_k) -\int h(w_k) \right| \Tend k
 \infty 0.
\]
Recalling that
\[
\int |u_{n_k}|^2 -\int|v_k|^2- \int|w_k|^2\Tend k \infty 0,
\]
we obtain
\begin{equation*}
  \int |u_{n_k}|^4\ln\(\frac{|u_{n_k}|^2}{\sqrt e}\)
-\int|v_k|^4\ln\(\frac{|v_{k}|^2}{\sqrt e}\) - 
\int|w_k|^4\ln\(\frac{|w_{k}|^2}{\sqrt e}\)\Tend k \infty 0. 
\end{equation*}
We consequently infer from \eqref{eq:kinetic} that
\begin{equation*}
  \liminf_{k\to \infty}\(E\(u_{n_k}\)-E(v_k)-E(w_k)\)\ge 0,
\end{equation*}
and thus
\begin{equation}\label{eq:lower}
  \limsup_{k\to \infty} \(E(v_k)+E(w_k)\)\le -\nu.
\end{equation}
Following an idea from \cite{CoJeSq10}, we now use a scaling argument and set
\begin{align*}
  \tilde v_k(x) & = v_k\(\sigma_k^{-1/2}x\),\quad  \sigma_k =
                  \frac{\rho}{\|v_k\|_{L^2}^2} \\
  \tilde w_k(x)& = w_k\(\mu_k^{-1/2}x\),\quad \mu_k =
                  \frac{\rho}{\|w_k\|_{L^2}^2} .
\end{align*}
We have $M( \tilde v_k)=M(\tilde w_k)=\rho $, and hence
$E(\tilde v_k), E(\tilde w_k)\ge -\nu.$ We also find that
\begin{equation*}
  E(\tilde v_k) = \sigma_k\(\frac{1}{2\sigma_k}\int |\nabla v_k|^2
  -\frac{\lambda}{2}\int |v_k|^4 \ln \( \frac{|v_k|^2}{\sqrt{e}}\)\),
\end{equation*}
and so
\begin{equation*}
  E(v_k) = \frac{1}{\sigma_k}E(\tilde v_k)
  +\frac{1-\sigma_k^{-1}}{2}\int |\nabla v_k|^2\ge
  \frac{-\nu}{\sigma_k}+\frac{1-\sigma_k^{-1}}{2}\int |\nabla
  v_k|^2. 
\end{equation*}
Doing the same for $E(w_k)$, yields
\begin{align*}
  E(v_k) + E(w_k) &\ge -\nu\( \frac{1}{\sigma_k}+\frac{1}{\mu_k}\)
  +\frac{1-\sigma_k^{-1}}{2}\int |\nabla  v_k|^2
                    +\frac{1-\mu_k^{-1}}{2}\int |\nabla  w_k|^2\\
  &\ge -\nu\( \frac{1}{\sigma_k}+\frac{1}{\mu_k}\)
  +\frac{1-\sigma_k^{-1}}{2\|v_k\|_{L^2}^2}\| v_k\|_{L^4}^4
                    +\frac{1-\mu_k^{-1}}{2\|w_k\|_{L^2}^2}\| w_k\|_{L^4}^4,
\end{align*}
where in the second step, we have used the Gagliardo-Nirenberg inequality. Passing to the
limit, we infer
\begin{align*}
  \lim\inf_{k\to \infty} \( E(v_k) + E(w_k) \)\ge -\nu
  +\frac{1}{2}\min\(\frac{1-\theta}{\theta\rho},
  \frac{\theta}{(1-\theta)\rho}\) \liminf_{k\to \infty}  \|u_{n_k}\|_{L^4}^4, 
\end{align*}
for any $\theta\in (0,1)$. By H\"older's inequality $\| u \|_{L^3}^3 \le \|u\|_{L^4}^2 \| u \|_{L^2}$ and thus, 
in view of Step 1 and the fact that $\|u_{n_k}\|_{L^2}^2=\rho>0$, we infer $$\liminf_{k\to \infty}\|u_{n_k}\|_{L^4}^4>0.$$ 
This is in contradiction to \eqref{eq:lower} and consequently rules out dichotomy.
\smallbreak

\noindent{\bf Step 3.} We can now invoke  \cite[Proposition 1.7.6(i)]{CazCourant} to deduce that for $u\in H^1(\R^2)$ and $(y_k)\subset \R^2$: 
$u_{n_k}(\cdot - y_k)\to u$ in $L^p(\R^2)$ for all $2\le p<\infty$. Together with the weak lower semicontinuity of the $H^1$ norm 
and the usual bound on the nonlinear potential energy, this implies 
\[
E(u) \le \lim_{k\to \infty}E(u_{n_k})=-\nu,
\]
and thus, the existence of a constraint energy minimizer.
\smallbreak

\noindent{\bf Step 4.} The orbital stability now follows by invoking classical arguments of \cite{CaLi82} (see also \cite{CazCourant}): 
Assume, by contradiction, that
there exist a sequence of initial data $(u_{0,n})_{n\in \N}\subset H^1(\R^2)$, such that
\begin{equation}\label{eq:8.3.17}
  \|u_{0,n}-\phi\|_{H^1}\Tend n \infty 0,
\end{equation}
and a sequence $(t_n)_{n\in \N}\subset \R$, such that the sequence of solutions $u_n$ to \eqref{eq:nls} associated to the initial data $u_{0,n}$ satisfies
\begin{equation}\label{eq:8.3.18}
  \inf_{\varphi\in \mathcal E(\rho)}\left\|u_n(t_n, \cdot) -
    \varphi\right\|_{H^1}>\eps,
\end{equation}
for some $\eps>0$.
Denoting $v_n=u_n(t_n, \cdot)$, the above inequality reads
\begin{equation*}
  \inf_{\varphi\in \mathcal E(\rho)}\|v_n-\varphi\|_{H^1}>\eps.
\end{equation*}
In view of \eqref{eq:8.3.17}, we find that, one the one hand:
\begin{equation*}
  \int_{\R^2}|u_{0,n}|^2 \Tend n \infty \int_{\R^2}|\phi|^2,\quad
  E\(u_{0,n}\)\Tend n \infty E(\phi)=\inf_{v\in \Gamma(\rho)}E(v).
  \end{equation*}
One the other hand, the conservation laws for mass and energy imply
\begin{equation*}
 M(v_n) \Tend n \infty M(\phi),\quad
  E\(v_{n}\)\Tend n \infty E(\phi),
\end{equation*}
and thus, $(v_n)_n$ is a minimizing sequence for the problem
\eqref{eq:8.3.5}. From the previous steps,  there exist a
subsequence, still denoted by $(u_n)_{n\in \N}$, and a sequence of points
    $(y_n)_{n\in \N}\subset \R^2$, such that $v_n(\cdot -y_n)$ has a strong limit $u$
    in $H^1(\R^2)$. In particular, $u$ satisfies \eqref{eq:8.3.5},
    hence a contradiction. 
\end{proof}

\appendix

\section{On the 1D case}
\label{sec:appendix}

Since the $L^2$-critical case in 1D requires a quintic nonlinearity, the
formal analogue of \eqref{eq:nls} reads
\begin{equation}
  \label{eq:1D}
  \left\{
\begin{aligned}
 & i\d_t u +\frac{1}{2}\d_x^2 u =\lambda u|u|^4\ln |u|^2,\quad x\in
 \R,\ \lambda>0,\\
  & u(0,x) = u_0\in H^1(\R).
\end{aligned}  
\right.
\end{equation}
Even though, to our knowledge, this model is not motivated by physics,
it is mathematically similar and gives a hint of what could be
expected for \eqref{eq:nls}. 
Global well-posedness follows from the same
arguments as in Theorem~\ref{thm:gwp}. The analogue of
Theorem~\ref{thm:ground} is also 
straightforward, and yields the condition $0<\omega<\tfrac{\lambda}{6e^{1/3}}$, since, in view of \cite[Theorem~5]{BL83a}, we compute
\begin{equation*}
  G(z) = -\omega z^2 -\frac{\lambda}{3}z^6 \ln\frac{z^2}{e^{1/3}}.
\end{equation*}
We then have a stronger notion of orbital stability than in the case of
Theorem~\ref{thm:stab}:
\begin{theorem}\label{theo:stab1D}
  Let $0<\omega<\tfrac{\lambda}{6e^{1/3}}$, and $\phi_\omega$ be the
  unique even and positive solution to
  \begin{equation*}
    -\frac{1}{2} \phi_\omega^{\prime\prime} +\lambda \phi_\omega
    |\phi_\omega|^4\ln|\phi_\omega|^2 +\omega\phi_\omega=0,\quad x\in \R.
  \end{equation*}
  Then, for all $\eps>0$, there exists $\delta>0$ such that if $u_0\in
  H^1(\R)$ satisfies
  \[\|u_0-\phi\|_{H^1(\R)}\le \delta,\]
  the
  solution to \eqref{eq:1D} satisfies
  \begin{equation*}
    \sup_{t\in \R}\inf_{{\theta\in \R}\atop{y\in
      \R}}\left\|u(t, \cdot)-e^{i\theta}\phi_\omega(\cdot
      -y)\right\|_{H^1(\R)}\le \eps.
  \end{equation*}
\end{theorem}
\begin{proof}
  The proof relies on the Grillakis-Shatah-Strauss theory \cite{GSS87},
  following the breakthrough of M.~Weinstein \cite{Weinstein86CPAM} (see also
  \cite{BGR15}), which implies that result is proven if we know that 
  $d(\omega):=S(\phi_\omega)$ is strictly convex, or, equivalently, that
  $M(\phi_\omega)$ is strictly increasing. Taking advantage of the
  one-dimensional setting, Iliev and Kirchev \cite[Lemma~6]{IlievKirchev93}
  have shown that
  \begin{equation*}
    d^{\prime\prime}(\omega)= -\frac{1}{2W'(a)}\int_0^a\(3 +
    \frac{as(f(a)-f(s))}{ag(s)-sg(a)} \)
      \(\frac{s}{W(s)}\)^{1/2} ds,
    \end{equation*}
    where, in the present case,
    \begin{align*}
  f(s) = \lambda s^2\ln s,\quad g(s) = \int_0^s f(\si)d\si
     =\frac{\lambda}{3}s^3 \ln\(\frac{s}{e^{1/3}}\),\quad  W(s) = \omega s + g(s),
    \end{align*}
    and $a$ is such that $W(a)=0$, $W'(a)<0$, and $W(s)>0$ for
    $0<s<a$.
    The existence of such an $a$ follows from direct computations. Note that
    \begin{equation*}
    3 +
    \frac{as(f(a)-f(s))}{ag(s)-sg(a)}  = \frac{3\frac{g(s)}{s}-
    3\frac{g(a)}{a} +f(a)-f(s)}{\frac{g(s)}{s}-\frac{g(a)}{a}}= \frac{1}{3}\frac{a^2-s^2}{\frac{g(s)}{s}-\frac{g(a)}{a}}.
\end{equation*}
By definition, $g(a)=-a\omega$, so
$\tfrac{g(s)}{s}-\tfrac{g(a)}{a}= \tfrac{W(s)}{s}$, and
\begin{equation*}
    d^{\prime\prime}(\omega)= -\frac{1}{2W'(a)}\int_0^a
      \frac{a^2-s^2}{W(s)}s\(\frac{s}{W(s)}\)^{1/2} ds,
    \end{equation*}
  Now the integrand is clearly nonnegative, and since $W'(a)<0$,
  $d(\omega)$ is strictly convex.
\end{proof}
%%%%%%%%%%%%%%%%%%%%%%%%%%%%%%%%%%%%%%%%%%%%%%%%%%%%%%%%%%%%
%%%%%%%%%%%%%%%%%%%%%%%%%%%%%%%%%%%%%%%%%%%%%%%%%%%%%%%%%%%%

\bibliographystyle{siam}

\bibliography{biblio}

%%%%%%%%%%%%%%%%%%%%%%%%%%%%%%%%%%%%%%%%%%%%%%%%%%%%%%%%%%%%

\end{document}